\documentclass[11pt]{amsart}
\usepackage{amsmath}
\usepackage{epic,eepic}	
\usepackage{hyperref} 

\date{\today}

\numberwithin{equation}{section}   

\theoremstyle{plain}

\newtheorem{thm}{Theorem}[section]

\newtheorem{cor}[thm]{Corollary}

\theoremstyle{definition}

\theoremstyle{remark}
\newtheorem{remark}[thm]{Remark}

\newcommand{\X}{\mathbf{X}}
\newcommand{\cn}{\colon}

\title{The Gronwall Inequality}

\author{Ralph Howard}
\address{Department of Mathematics,
University of South Carolina,
Columbia, SC 29208}
\email{howard@math.sc.edu}
\urladdr{www.math.sc.edu/$\sim$howard}

\begin{document}
\maketitle

\section{Introduction.}

The Gronwall inequality as given here estimates the difference of
solutions to two differential equations $y'(t)=f(t,y(t))$ and
$z'(t)=g(t,z(t))$ in terms of the difference
between the initial conditions for the equations and the difference
between $f$ and $g$.  The usual version of the inequality is when
$f=g$, but there is little extra work involved in proving the more general
case \cite{Gronwall}.
\medskip

This note started life as note/homework problem on a web page for
an ordinary differential equation class in 1998.  Since then
it has been referenced several times: 
\cite{Bourgain-Klein,Klein-Tsang,ABBDDR,SQWZT,Srikrishnan-Chaudhuri}. 
So it seemed worthwhile to put a copy on the
arXiv so that there is a version more permanent than
one on a personal web page.  While I do not know of an
explicit reference to the inequality here, I very much doubt
it is original to me.

\section{The Inequality}

\begin{thm}[The Gronwall Inequality]\label{thm:gronwall}
Let $\X$ be a Banach space and $U\subset \X$ an open convex set in
$\X$.  Let $f,g\cn [a,b]\times U\to \X$ be continuous functions and
let $y,z\cn [a,b]\to U$ satisfy the initial value problems
\begin{align}
y'(t)&=f(t,y(t)),\quad y(a)=y_0,\\
z'(t)&=g(t,z(t)),\quad z(a)=z_0.
\end{align}
Also assume there is a constant $C\ge 0$ so that
\begin{equation}\label{lip}
\|g(t,x_2)-g(t,x_1)\|\le C\|x_2-x_1\|
\end{equation}
and a continuous function $\phi\cn [a,b]\to [0,\infty)$ 
so that
\begin{equation}\label{f-g}
\|f(t,y(t))-g(t,y(t))\|\le \phi(t).
\end{equation}
Then for $t\in [a,b]$
\begin{equation}\label{gronwall}
\|y(t)-z(t)\|
 \le e^{C|t-a|}\|y_0-z_0\|+e^{C|t-a|}\int_a^te^{-C|s-a|}\phi(s)\,ds.
\end{equation}
\end{thm}

\begin{remark}
The inequality~(\ref{f-g}) is a little awkward as it involves an
inequality along the solution $y(t)$ which we may not know.  But
we can replace~(\ref{f-g}) by the stronger hypothesis
\begin{equation}\label{uniform-f-g}
\|f(t,x)-g(t,x)\|\le \phi(t)\quad \text{for all} \quad x\in U
\end{equation}
(which clearly implies~(\ref{f-g})) and get the same result.  Of course
this may mean using a choice of $\phi$ that is larger than is needed
in~(\ref{f-g}).\qed  
\end{remark}

\begin{remark}
Note we are not assuming $f$ satisfies a Lipschitz condition and
therefore solutions to the initial value problem $y'(t)=f(t,y(t))$,
$y(a)=y_0$ need not be unique.  In this case the inequality can be
used to estimate $y(t)$ by comparing it to the solution to
$z'(t)=g(t,z(t))$, $z(a)=y_0$ with the same initial condition.  In
this case the inequality becomes
$$
\|y(t)-z(t)\| \le e^{C|t-a|}\int_a^te^{-C|s-a|}\phi(s)\,ds.
$$
Then if we assume that the inequality~(\ref{uniform-f-g}) holds then
and $y_1'(t)=f(t,y_1(t))$ with $y_1(a)=y_1$ then using the obvious
inequality $\|y(t)-y_1(t)\|\le \|y(t)-z(t)\|+\|z(t)-y_1(t)\|$ where
$z'(t)=g(t,z(t))$, $z(a)=y_0$ we get
the inequality
$$
\|y(t)-y_1(t)\|\le
e^{C|t-a|}\|y_0-y_1\|+2e^{C|t-a|}\int_a^te^{-C|s-a|}\phi(s)\,ds. 
$$
This gives a version of the Gronwall inequality for differential
equations that do not satisfy a Lipschitz condition in terms of the
Lipschitz constant of a ``nearby'' differential equation that does
(where near is ``nearby'' is measured by the size of $\phi$).\qed
\end{remark}

\begin{proof}[Proof of Theorem~\ref{gronwall}]  
Let $x\cn [a,b]\to \X$ be continuously differentiable.
Then the map $t\mapsto \|x(t)\|$ is Lipschitz and 
therefore absolutely continuous.  Thus it is differentiable
almost everywhere, the Fundamental Theorem of Calculus applies,
and it is easily checked that $ \dfrac{d}{dt} \|x'(t)\|\le \| x'(t)\|$.
  Whence, setting $x(t) = y(t) - y(t)$,  using the
hypothesis~(\ref{f-g}), and~(\ref{lip}),
\begin{align*}
\dfrac{d}{dt}\|y(t)-z(t)\|&\le \|y'(t)-z'(t)\|\\
	&=\|f(t,y(t))-g(t,z(t))\|\\
	&\le \|f(t,y(t))-g(t,y(t))\|+\|g(t,y(t))-g(t,z(t))\|\\
	&\le \phi(t) +C\|y(t)-z(t)\|.
\end{align*}
That is 
$$
\dfrac{d}{dt}\|y(t)-z(t)\|-C\|y(t)-z(t)\|\le \phi(t).
$$
Multiplication by the integrating factor $e^{-Ct}$ yields
$$
\dfrac{d}{dt}\left(e^{-Ct}\|y(t)-z(t)\|\right)\le e^{-Ct}\phi(t).
$$
Integrate this from $a$ to $t$ to get
$$
e^{-Ct}\|y(t)-z(t)\|-\|y_0-z_0\|\le \int_0^te^{-Cs}\phi(s)\,ds.
$$
This is equivalent to~(\ref{gronwall}).
\end{proof}

The following is the standard form of the Gronwall inequality.

\begin{cor}
Let $\X$ be a Banach space and $U\subset \X$ an open convex set in
$\X$.  Let $f\cn [a,b]\times U\to \X$ be a continuous function and
let $y,z\cn [a,b]\to U$ satisfy the initial value problems
\begin{align}
y'(t)&=f(t,y(t)),\quad y(a)=y_0,\\
z'(t)&=f(t,z(t)),\quad z(a)=z_0.
\end{align}
Also assume there is a constant $C\ge 0$ so that
\begin{equation}\label{std:lip}
\|f(t,x_2)-f(t,x_1)\|\le C\|x_2-x_1\|.
\end{equation}
Then for $t\in [a,b]$
\begin{equation}\label{std:gronwall}
\|y(t)-z(t)\|
 \le e^{C|t-a|}\|y_0-z_0\|.
\end{equation}
\end{cor}

\begin{proof}
In Theorem~\ref{thm:gronwall} let $f=g$.  Then we can take
$\phi(t)\equiv0$ in~(\ref{f-g}).  Then~(\ref{gronwall}) reduces
to~(\ref{std:gronwall}).
\end{proof}

\section*{Acknowledgement} 
I wish to thank Josh Zahl for encouraging me 
put this version on the axXiv.

\providecommand{\bysame}{\leavevmode\hbox to3em{\hrulefill}\thinspace}
\providecommand{\MR}{\relax\ifhmode\unskip\space\fi MR }
\providecommand{\MRhref}[2]{%
  \href{http://www.ams.org/mathscinet-getitem?mr=#1}{#2}
}
\providecommand{\href}[2]{#2}

\end{document}